\documentclass{amsart}[12pt,article]

\usepackage{amsmath,amssymb,amscd,amsthm,indentfirst}
\usepackage[all]{xy}
\usepackage{enumerate}
\usepackage{hyperref}
\usepackage{graphicx}

\newtheorem{thm}{Theorem}[section]

\newtheorem{cor}[thm]{Corollary}
\newtheorem{lem}[thm]{Lemma}

\newtheorem*{notation}{Notation}

\theoremstyle{definition}

\newtheorem{exa}[thm]{Example}

\def\Out{\mathrm{Out}}

\def\Aut{\mathrm{Aut}}

\newcommand{\Syl}{\operatorname{Syl}\nolimits}

\makeatletter
\@tempcnta=\@ne
\def\@nameedef#1{\expandafter\edef\csname #1\endcsname}
\loop\ifnum\@tempcnta<27
  \@nameedef{C\@Alph\@tempcnta}{\noexpand\mathcal{\@Alph\@tempcnta}}
  \advance\@tempcnta\@ne
\repeat

\makeatletter
\@tempcnta=\@ne
\def\@nameedef#1{\expandafter\edef\csname #1\endcsname}
\loop\ifnum\@tempcnta<27
  \@nameedef{B\@Alph\@tempcnta}{\noexpand\mathbb{\@Alph\@tempcnta}}
  \advance\@tempcnta\@ne
\repeat

\makeatletter
\@tempcnta=\@ne
\def\@nameedef#1{\expandafter\edef\csname #1\endcsname}
\loop\ifnum\@tempcnta<27
  \@nameedef{F\@Alph\@tempcnta}{\noexpand\mathsf{\@Alph\@tempcnta}}
  \advance\@tempcnta\@ne
\repeat

\date{\today}
\title{A $p$-nilpotency criterion for finite groups}

\author{Antonio D\'{i}az Ramos}
\author{Antonio Viruel}
\address{Departamento de {\'A}lgebra, Geometr{\'\i}a y Topolog{\'\i}a,
Universidad de M{\'a}\-la\-ga, Apdo correos 59, 29080 M{\'a}laga,
Spain.}
\email{adiazramos@uma.es\\
viruel@uma.es}
\thanks{Authors partially supported by MEC grant MTM2016-78647-P and Junta de Andaluc{\'\i}a grant FQM-213.}

\begin{document}

\begin{abstract}
We prove a $p$-nilpotency criterion for finite groups in terms of the element orders of its $p'$-reduced sections that extends a nilpotency criterion by T{\u{a}}rn{\u{a}}uceanu.
\end{abstract}

\maketitle

\section{Introduction}
\label{section:introduction}

In \cite{Tarnauceanu2018}, T{\u{a}}rn{\u{a}uceanu introduces a nilpotency criterion for finite groups by studying element orders. This criterion involves the following generalization of the Euler’s totient function \cite{Tarnauceanu2015},
\[
\varphi(G)=|\{g\in G|o(g)=\exp(G)\}|,
\]
where $G$ is a finite group. Obviously, if $G$ is nilpotent, then $\varphi(G)$ is necessarily different from zero. Moreover, subgroups and quotients of a nilpotent group are also nilpotent, so $\varphi(K)\neq 0$ for every section $K$ of the nilpotent group $G$. T\u arn\u auceanu proves that the converse holds, i.e., if $\varphi(K)\neq 0$ for all sections $K$ of $G$, then $G$ is nilpotent.

In this note, we give the following related criterion for $p$-nilpotency of finite groups, in which we consider instead $p'$-reduced sections.

\begin{thm}\label{thm:main}
Let $G$ be a finite group and let $p$ be a prime. Then the following are equivalent:
\begin{enumerate}
\item[(1)] $G$ is $p$-nilpotent.
\item[(2)] $\varphi(K)\neq 0$ for each $p'$-reduced section $K$ of $G$.
\item[(3)] $\varphi(K)\neq 0$ for each $p'$-reduced section $K$ of $N_G(Q)$, where $Q$ is any $p$-centric $p$-subgroup of $G$.
\item[(4)] $\varphi(K)\neq 0$ for each $p'$-reduced section $K$ of $N_G(Q)$, where $Q$ is any $p$-centric and $p$-radical $p$-subgroup of $G$.
\end{enumerate}
\end{thm}

Recall that $Q$ is $p$-centric if $Z(P)\in\Syl_p(C_G(P))$ and that $Q$ is $p$-radical if $N_G(P)/P$ is $p$-reduced. Our proof is formulated in the modern language of fusion systems \cite{BLO2}, following the trend of the works \cite{BESW}, \cite{BGH}, \cite{CSV}, \cite{D}, \cite{DGPS}, \cite{DEV}, \cite{GRV}, \cite{KLN} and \cite{LZ}. Our proof of Theorem \ref{thm:main}, carried out in the following section, is independent of the classification of minimal non-nilpotent groups. Therefore, Theorem \ref{thm:main} provides an alternative proof of the main result in \cite{Tarnauceanu2018}.

\begin{cor}\label{cor:main}
Let $G$ be a finite group. Then the following are equivalent:
\begin{enumerate}
\item[(1)] $G$ is nilpotent.
\item[(2)] $\varphi(K)\neq 0$ for each section $K$ of $G$.
\end{enumerate}
\end{cor}

Regarding Theorem \ref{thm:main}$(2)$, one could wonder whether $\varphi(K)\neq 0$ for the particular $p'$-reduced section $K=G/O_{p'}(G)$ is enough for the $p$-nilpotency of $G$. As the next example shows, this is not the case.

\begin{exa}
We exhibit a finite group $G$ such that $O_{p'}(G)=1$, $\varphi(G)\neq 0$ and $G$ is not $p$-nilpotent.

We start considering a group $K$ such that $K$ is $p'$-reduced and $\varphi(K)=0$. Such group can be obtained as a semidirect product via Lemma \ref{lemma:1} below. Then $K$ is not $p$-nilpotent by Theorem \ref{thm:main}$(2)$. If $\exp(K)=p_1^{e_1}\cdots p_s^{e_s}$ with $p_i$'s different primes, there exist, for each $i$, an element $k_i\in K$ with $o(k_i)=p_i^{e_i}$. Define $G$ as the direct product of $s$ copies of $K$. Then clearly $\exp(G)=\exp(K)=o((k_1,\ldots,k_s))$ and hence $\varphi(G)\neq 0$. Moreover, $O_{p'}(G)=1$ and $G$ is not $p$-nilpotent as  $K$ is neither.
\end{exa}

\begin{notation} We denote by $O_{p'}(G)$ the largest normal subgroup of $G$ of order prime to $p$, and by $O_p(G)$ the largest normal $p$-subgroup of $G$. We say that $G$ is $p'$-reduced if $O_{p'}(G)=1$ and that $G$ is $p$-reduced if $O_p(G)=1$. For terminology and results on fusion systems, we refer the reader to \cite{AKO} and \cite{Craven2011}.
\end{notation}

\section{Proof of the theorem}
\label{section:proof}

We need two preliminary results.

\begin{lem}\label{lemma:1}
Let $p$ and $q$ be different primes and let $V$ be a non-trivial simple $\BF_p[\BZ/q]$-module. Then $\varphi(V\rtimes \BZ/q)=0$.
\end{lem}

\begin{proof}
Note that $\exp(V\rtimes \BZ/q)=pq$. Assume that there exists $x\in V\rtimes \BZ/q$ satisfying $o(x)=pq$. Then $x^q\in V$ and $x^p$ centralizes $x^q$. As $x^p$ projects onto a generator of $\BZ/q$, $\BF_p\cong\langle x^q\rangle\leq V$ is a trivial $\BF_p[\BZ/q]$-submodule, which is not possible by hypothesis.
\end{proof}

\begin{lem}\label{lemma:2}
Consider a short exact sequence of groups,
\[
\xymatrix{
P\ar[r]& G \ar[r]^\pi & K,
}
\]
where $1\neq P$ is a $p$-group, $C_G(P)\leq P$ and $1\neq K$ is $p$-reduced. Then there exists a $p'$-reduced section $H$ of $G$ with $\varphi(H)=0$.
\end{lem}

\begin{proof}
As $K$ is non-trivial and $p$-reduced, there exists a prime $1<q\neq p$ such that $q$ divides $|K|$. Hence, we can choose $g\in G$ such that $o(\pi(g))=q$. Moreover, raising $g$ to an appropriate power of $p$, we can assume that $o(g)=q$.

The automorphism induced by $g$ on $P$, $c_g\in \Aut(P)$, cannot be trivial because $C_G(P)\leq P$ and $g\notin P$. Hence, $c_g$ must have order $q$. By Burnside's theorem on coprime actions \cite[Theorem 5.1.4]{Gorenstein}, $c_g$ induces a non-trivial order $q$ automorphism on the Frattini quotient $W=P/\Phi(P)\cong (\BF_p)^r$ ($r\geq 1$). Hence, $W$ is a non-trivial $\BF_p[\BZ/q]$-module.

By Maschke's theorem \cite[Theorem 3.3.1]{Gorenstein}, we can decompose $W$ as a direct sum of simple $\BF_p[\BZ/q]$-submodules. As $W$ is a non-trivial $\BF_p[\BZ/q]$-module, one of these simple submodules is non-trivial. Call it $V$. Then $H=V\rtimes \BZ/q$ is a $p'$-reduced section of $G$ and $\varphi(H)=0$ by Lemma \ref{lemma:1}.
\end{proof}

Now we can prove the main theorem.

\begin{proof}[Proof of Theorem \ref{thm:main}]
To prove $(1)\Rightarrow(2)$, let $G$ be a $p$-nilpotent group and recall that $p$-nilpotency is preserved under subgroups and quotients. Hence, any section $K$ of $G$ is $p$-nilpotent and may be written as $K=O_{p'}(K)S$, with $S\in \Syl_p(K)$. If $K$ is $p'$-reduced, i.e., $O_{p'}(K)=1$, we must have $K=S$ and hence $\varphi(K)\neq 0$.

The remaining non-trivial implication is $(4)\Rightarrow (1)$. We use Frobenius's normal $p$-complement theorem \cite[Theorem 7.4.5]{Gorenstein} in its version for fusion systems \cite[Theorem 1.12]{Craven2011}. Thus, $G$ is $p$-nilpotent if and only if $\CF_S(G)=\CF_S(S)$, where $S\in\Syl_p(G)$ and $\CF=\CF_S(G)$ is the fusion system of $G$ over $S$. In turn, by Alperin's fusion theorem for fusion systems \cite[Theorem A.10]{BLO2}, to show that $\CF=\CF_S(S)$, it is enough to show that $S$ is the only $\CF$-centric and $\CF$-radical subgroup of $S$, and that $\Out_\CF(S)=1$.

So let $Q\leq S$ be $\CF$-centric and $\CF$-radical. Then $C_G(P)=Z(P)\times C'_G(P)$ with $C'_G(P)=O_{p'}(C_G(P))$ \cite[Lemma A.4]{BLO1}, $O_p(\Out_G(Q))=1$, and $Q$ is $p$-centric and $p$-radical. Consider the short exact sequence,
\[
\xymatrix{
Q\ar[r]& N_G(Q)/C'_G(Q) \ar[r] & \Out_G(Q)=N_G(Q)/QC_G(Q).
}
\]
It satisfies the hypothesis of Lemma \ref{lemma:2} unless $\Out_G(Q)=1$.  In the former case, we get a contradiction with hypothesis $(4)$. So $\Out_G(Q)=1$. As $\Out_S(Q)\leq \Out_G(Q)$, this cannot be the case for $Q<S$. So $Q=S$ and $\Out_\CF(S)=\Out_G(S)=1$.
\end{proof}

\begin{proof}[Proof of Corollary \ref{cor:main}]
By the comments in the introduction, it is left to prove that $(2)$ implies $(1)$. Recall that $G$ is nilpotent if and only it is $p$-nilpotent for every prime $p$. Finally, since every $p'$-reduced section of $G$ is obviously a section of $G$, then $(2)$ implies Theorem \ref{thm:main}$(2)$ and therefore $G$ is $p$-nilpotent for every prime $p$.
\end{proof}

\end{document}